\documentclass[twoside,11pt]{article}

\usepackage{amssymb, amsmath, latexsym, mathrsfs, verbatim, calc}
\usepackage{amsthm}
\usepackage{amssymb}

\usepackage{graphics,color,graphicx,shortvrb}

\usepackage[latin1]{inputenc}   

\def \N{{\rm I\!N}}
\def\i{{\bf 1}}

\def \R{{\mathop{{\rm I\negthinspace R}}}}

\def\Dt0{{\bf D}(t_0)}

\def\to{\rightarrow}

\def\bF{{\bf F}}

\newcommand{\ba}{\[\begin{array}{rl}}
\newcommand{\ea}{\end{array}\]}
\newcommand{\be}{\begin{equation}}
\newcommand{\ee}{\end{equation}}
\newcommand{\bea}{\begin{eqnarray}}
\newcommand{\eea}{\end{eqnarray}}
\newcommand{\beaa}{\begin{eqnarray*}}
\newcommand{\eeaa}{\end{eqnarray*}}

\newtheorem{thm}{Theorem}[section]
\newtheorem{Lemma}[thm]{Lemma}

\newtheorem{prop}[thm]{Proposition}

\newtheorem{Definition}[thm]{Definition}

\newcommand{\AR}{{\cal A}}
\newcommand{\BR}{{\cal B}}
\newcommand{\CR}{{\cal C}}

\newcommand{\FR}{{\cal F}}

\newcommand{\NR}{{\cal N}}

\newcommand{\SR}{{\cal S}}
\newcommand{\UR}{{\cal U}}
\newcommand{\VR}{{\cal V}}

\begin{document}
\title{Pathwise strategies for stochastic differential games\\ 
\large with an erratum to ``Stochastic Differential Games with Asymmetric Information''}
\author{P. Cardaliaguet\thanks{Ceremade, Universit\'e Paris-Dauphine,
Place du Mar\'echal de Lattre de Tassigny, 75775 Paris cedex 16 (France). e-mail:  cardaliaguet@ceremade.dauphine.fr}  \and
C. Rainer\thanks{Universit\'e de Bretagne Occidentale, 6, avenue Victor-le-Gorgeu, B.P. 809, 29285 Brest cedex, France.
e-mail: Catherine.Rainer@univ-brest.fr}}
\maketitle
\begin{abstract} We introduce a new notion of pathwise strategies for stochastic differential games. This allows us to give a correct meaning to some statement asserted in \cite{sdgai}. 
\end{abstract}

\section{Introduction.}

In this short note we develop a new notion of  strategies for stochastic differential games. 
We present our concept in the framework of two-player, zero-sum, differential games. The players are labelled Player I and Player II, Player I being minimizing while Player II is maximizing.
We assume that the players have perfect monitoring, i.e., they observe each other's action perfectly. The state of the game satisfies a stochastic differential equation, that we assume driven by a Brownian motion, and which is controlled by both players. 

Nonanticipative strategies for  deterministic differential games have been introduced in a series of papers by Varaiya \cite{VAR167}, Roxin \cite{ROX69}, Elliott and Kalton \cite{ELKA72}: in this framework a strategy (for Player I) is a nonanticipative map from Player II's set of controls to Player I's. Adapting this idea to stochastic differential games lead  Fleming and Souganidis in their pioneering work \cite{flemingsouganidis}  to define a notion of strategy (again for Player I) as nonanticipative map from the set of {\it adapted} controls of Player II to the set of {\it adapted} controls of Player I. This approach has subsequently been used by most authors working on stochastic differential games, sometimes with some variants: see, e.g., Buckdahn-Li \cite{buckli}.  

If the notion of nonanticipative strategies makes perfectly sense for {\it deterministic} differential games---because the players indeed observe each other's action---one can object that, for {\it stochastic} ones,   the players do not actually observe their opponent's {\it adapted} control, but just a {\it realization} of this control in the actual state of the world: more precisely, assume that Player II plays the control $v=v(t,\omega)$. Then, in the state $\omega$ and at time $t$, Player I has not observed the full map $(s,\omega')\to v(s,\omega'_{[0,t]})$, but only the map $s\to v(t,\omega_{[0,t]})$. For this reason, the authors of the present paper introduced in  \cite{sdgai} a notion of {\it pathwise} nonanticipative strategies, formalizing the fact that the players only observe their opponent's action in the actual state, as well as the path of the resulting solution of the stochastic differential equation. 

Unfortunately handling such pathwise strategies is quite subtle and, in \cite{sdgai}, we overlooked some difficulties 
(we explain this in details in section \ref{sec:3}). 
In the present paper we show how to overcome this problem. We  still keep the flavor of pathwise strategies, but require the stronger condition that the players observe the control actually played by their opponent {\it as well as  the Brownian path} (instead of the trajectory of SDE). The key point is that the players can nevertheless deduce the resulting solution of the SDE: to show this we use the pathwise construction of stochastic integrals introduced  by Nutz \cite{Nutz}.

This note is divided into two parts: first we introduce the new notion of strategies and show the existence of a value and its characterization for a classical two-player zero-sum game with a final cost (to better explain  our ideas, we have chosen to present our approach in this simple framework).  The second part of the note is devoted to the erratum of the paper \cite{sdgai}.\\

\noindent {\bf Acknowledgment :}  We are very much indebted with Rainer Buckdahn, who pointed out to us the flaw in the paper \cite{sdgai}. 

This work has been partially supported by the Commission of the
European Communities under the 7-th Framework Programme Marie
Curie Initial Training Networks   Project SADCO,
FP7-PEOPLE-2010-ITN, No 264735, and by the French National Research Agency
 ANR-10-BLAN 0112.

\section{The classical stochastic differential game revisited.}
\label{defdef}

Let $T>0$ be a deterministic time horizon. For all $t\in[0,T]$,
let $\Omega_t$ be the set of continuous maps from $[t,T]$ to $\R^d$ endowed with the $\sigma$-algebra generated by the coordinate process and $P_t$, the Wiener measure on it. We denote by  $W$ the canonical process:  $W_s(\omega)=\omega(s)$. We introduce also the filtration $\bF_t=(\FR_{t,s}=\sigma\{ W_r-W_t,r\in[t,s]\}$, completed by all null sets of $P_t$.

For any $t\in [0,T]$ we denote by ${\mathcal C}^0([t,T],\R^N)$ the set of continuous maps from $[t,T]$ into $\R^N$ endowed with the sup norm and by ${\mathcal B}_t$ the associated Borel $\sigma-$algebra. 

The dynamic of the game is given by
\begin{equation}
\label{dynjeu}
\left\{\begin{array}{l}
dX_s=f(X_s,u_s,v_s)ds+\sigma(X_s,u_s,v_s)dW_s, s\in [t,T],\\
X_t=x\in \R^N,
\end{array}\right.
\end{equation}
with $u$ and $v$ two   $\bF_t$-adapted processes with values in some compact metric spaces $U$ and $V$. The process $u$ (resp. $v$) represents the action of Player I (resp. Player II).
We denote by $X^{t,x,u,v}$ the solution of (\ref{dynjeu}).

Throughout the paper, the maps $f:\R^N\times U\times V\to \R^N$ and $\sigma: R^N\times U\times V\to \R^{N\times d}$ are supposed to be bounded, continuous, Lipschitz continuous in $(t,x)$ uniformly with respect to $u,v$. The sets $U$ and $V$ are compact subsets of finite dimensional spaces. We denote by $U_{t}$ (resp. $V_{t}$)  the set of  measurable  maps from $[t,T]$ to $U$ (resp. $V$), while 
$\UR(t)$ (resp. $\VR(t)$) denotes the set of  $\bF_t$-adapted processes with values in $U$ (resp. $V$).
In what follows, the sets
$U_t$ and $V_t$ are  endowed with the   $L^1$-distance and the Borel $\sigma$-field generated by it.

\begin{Definition}\label{def:strategy}
A  strategy for Player I at time $t$ is a nonanticipative, Borel-measurable map $\alpha: \Omega_t\times V_{t}\to U_{t}$ with delay: there exists $\delta>0$ such that, for any $t'\in[t,T]$,  all $(v_1,v_2)\in V_{t}^2$ and for $P_t$-a.s. any $(\omega_1,\omega_2)\in\Omega_t^2$, if $(\omega_1,v_1)= (\omega_2,v_2)$ a.s. on $[t,t']$, then
$\alpha(\omega_1,v_1)= \alpha(\omega_2,v_2)$ a.s. on $[t,t'+\delta]$. 
Strategies for Player~II are defined in a symmetrical way. The set of strategies for Player I (resp. Player II) is denoted by $\AR(t)$ (resp. $\BR(t)$).
\end{Definition}

 Let us point out that, for all $ \alpha\in\AR(t)$ and $v\in\VR(t)$, $\alpha(v)$ is a process and belongs to $\UR(t)$. In the same way, for all $\beta\in\BR(t)$ and $u\in\UR(t)$, $\beta(u)$ belongs to $\VR(t)$. We denote by $\UR_d(t)$ the subset of $\UR(t)$ of controls $u\in\UR(t)$ for which there exists some $\delta>0$ such that, $(u_s)_{s\in[t,T]}$ is adapted to $\bF^\delta_t:=(\FR_{t,(s-\delta)\vee t})$. The set $\VR_d(t)$ is defined in a similar way. We remark that the elements of $\UR_d(t)$ and $\VR_d(t)$ are predictable for the original fitration $\bF_t$.\\
Now we can state our  fix point lemma: 

\begin{Lemma}\label{lem:fixpoint} For all $t\in[t,T]$, for all $(\alpha,\beta)\in\AR(t)\times\BR(t)$, there exists a unique pair of controls $(u,v)\in \UR_d(t)\times\VR_d(t)$ which satisfies, $P$-a.s.
\begin{equation}
\label{fix}
 u=\alpha(v),\;  v=\beta(u).
\end{equation}
\end{Lemma}

\begin{proof} Let $\delta>0$ be a common delay for $\alpha$ and $\beta$. We can choose $\delta$ such that $T=t+N\delta$, for some $N\in\N^*$. We define  on $\Omega_t$, $\UR_k$ (resp. $\VR_k$) the set of $\bF^\delta_t$-adapted processes on the time interval $[t,t+k\delta)$ with values in $U$ (resp. $V$). 

 By definition, on $[t,t+\delta)$, the control $\alpha(\omega,v)$ does not depend on $(\omega,v)$: we can set, for all $(\omega,v)\in\Omega_t\times V_t$, $\alpha(\omega,v)=u_0$, where $u_0\in\UR_0$ (in fact $u_0$ is deterministic). And in the same way, there exists $v_0\in\VR_0$ such that, for all $(\omega,u)\in\Omega_t\times U_t$, $\beta(\omega,u)=v_0$. 
 
Assume now that, for some $k\in\{ 0,\ldots,N-1\}$, there exists a pair $(u_k,v_k)\in\UR_k\times\VR_k$ such that, on $[t, t+k\delta)$, $\alpha(v_k)=u_k$ and $\beta(u_k)=v_k$ $P$-a.s. . We set, for all $\omega\in\Omega_t$, $u_{k+1}(\omega)=u_k(\omega)$ on $[t,t+k\delta)$ and, since $\alpha$ is non anticipative with delay $\delta$, it makes sense to set $u_{k+1}=\alpha(\omega,v_k(\omega))$ on $[t+k\delta,t+(k+1)\delta)$. By assumption, $u_k$ and $v_k$ are adapted to $\bF^\delta_t$  and $\alpha$ is nonanticipative. It follows that a $u_{k+1}$ is also adapted to $\bF^\delta_t$. The process  $v_{k+1}$ can be defined in a similar way.

At the end it is sufficient to set $(u,v)=(u_N,v_N)$ to get the desired result. 
\end{proof}

The main issue with our notion of strategies is that it is not clear that the observation of the brownian path and of the realized control of the opponent  up to some time $t'$ suffices to compute the position of the system at time $t'$. The following Lemma---which is the main point in our approach---says that this is actually the case. 

\begin{Lemma}
\label{lemf}
Fix $(t,x)\in[0,T]\times\R^N$. For all $\alpha_0\in\AR(t)$, there exists a measurable function $F:=F_{\alpha_0,t,x}:(V_t\times\Omega_t,\BR(V_t)\otimes\FR_t)\rightarrow({\mathcal C}^0([t,T],\R^N),\BR_t)$ such that, for all $\bar v\in\VR_d(t)$,
\[ F(\bar v(\cdot),\cdot)=X^{t,x,\alpha_0(\bar v),\bar v} \; \; P_t\mbox{-a.s.}\; .\]
Furthermore the map $F$ is nonanticipative, in the sense that there exists $\Omega'_t\subset\Omega_t$ with $P_t(\Omega_t')=1$ such that, for all $\omega,\omega'\in\Omega_t'$ and $v,v'\in V_t$, if $(v,\omega)=(v',\omega')$ a.s. on $[t,t']$, then, $F(v,\omega)=F(v',\omega')$ on $[t,t']$. 
%
\end{Lemma}

\begin{proof}
 Let $\bar v\in\VR_d(t)$ and $P^{\bar v}_{t}$ the law on $V_{t}\times\Omega_{t}$ of $(\bar v,W)$ under $P_{t}$. We endow the set  $V_{t}\times\Omega_{t}$ with the following filtration : let $V_{t,t'}$ be the set of measurable maps from $[t,t']$ to $V$ and set $\BR_{t,t'}=\{ \{v\in V_t,v_{[t,t']}\in B\}, B\in\BR(V_{t,t'})\}$.
Set $\displaystyle \FR^*_{t,s}=\bigcap_{\bar v\in\VR_d(t)}(\BR_{t,t'}\otimes\FR_{t,s}\vee\NR^{P^{\bar v}})$, with $\NR^{P^{\bar v}}$ the set of null sets for the probability $P^{\bar v}$. Then $(\FR_{t,s}^*,s\in[t,t'])$ is a filtration satisfying the usual assumptions and in which $(W_s(v,\omega):=\omega(s),s\in[t,t'])$ is a Brownian motion.\\
On the filtered probability space $(V_t\times\Omega_t ,\BR(V_t)\otimes\FR_t,P^{\bar v}_{t};  (\FR^*_{t,s},s\in[t,t']))$, we consider now the following SDE:
\begin{equation}
\label{sde}
\left\{ \begin{array}{l}
d\tilde X_s=b(\tilde X_s,\alpha_0({\bf v})_s,{\bf v}_s)ds+\sigma(\tilde X_s,\alpha_0({\bf v})_s,{\bf v}_s)dW_s,\; s\in[ t,t'],\\
\tilde X_{t}=x,
\end{array}\right.
\end{equation}
with 
\[ \left\{ \begin{array}{l}
W(v,\omega)_s=\omega(s),\\
{\bf v}(v,\omega)_s=v(s).
\end{array}\right.\]
Then (\ref{sde}) has a strong solution $\tilde X^{t,x,\alpha_0}$ which law, under $P^{\bar v}_{t}$ on $V_{t}\times\Omega_{t}$ coincides with the law of $X^{t,x,\alpha_0(\bar v),\bar v}$ under $P_{t}$ on $\Omega_{t}$.\\
We now  apply the main Theorem of \cite{Nutz} to the above filtered space and the processes
\[S_s=\left(\begin{array}{l}
s\\W_s\end{array}\right),\;\;  H_s=\left(\begin{array}{l}
b(X_s^{t,x,\alpha_0({\bf v}),{\bf v}},\alpha_0({\bf v})_s,{\bf v}_s)\\\sigma(X_s^{t,x,\alpha_0({\bf v}),{\bf v}},\alpha_0({\bf v})_s,{\bf v}_s)\end{array}\right).\]
We obtain that there exists a map $F:V_{t}\times\Omega_{t}\rightarrow{\mathcal C}^0([t,T],\R^N)$ which is adapted with respect to the filtration $(\FR^*_{t,s})$,  
such that, for all $\bar v\in\VR_d(t)$ and for all bounded test function $\varphi:\R^N\times\Omega_t\rightarrow\R$,
\[E_t[\varphi(F(\bar v(\cdot),\cdot),\cdot)]= E^{\bar v}_{t}[\varphi(F)]=E_{t}^{\bar v}[\varphi(\tilde X^{t,x,\alpha_0}_{t'})]=E_{t}[\varphi(X^{t,x,\alpha_0(\bar v),\bar v}_{t'})].\]
The nonanticipativity of $F$ follows from the fact that $F$ is adapted with respect to $(\FR^*_{t,s})$.
\end{proof}

We now define the value functions of the game. Given a bounded and Lipschitz continuous terminal cost $g:\R^N\rightarrow\R$,  an initial position $(t,x)\in[0,T]\times\R^N$ and a pair of adapted controls  $(u,v)\in\UR(t)\times\VR(t)$,  we define the cost function 
\[ J(t,x,u,v)=E_t\left[g(X^{t,x,u,v}_T)\right].\]
 It is well known that, for all pair of controls $(u,v)\in\UR(t)\times\VR(t)$,  $(s,x)\mapsto J(s,x,u,v)$ is Lipschitz in $x$ and H\"{o}lder in $s$, uniformly in $(u,v)$. It follows that, for all $(\alpha,v)\in\AR(t)\times\VR(t)$, $(s,x)\mapsto J(s,x,\alpha(v),v)$ is also Lipschitz continuous in $x$ and H\"{o}lder continuous  in $s$, uniformly in $\alpha$ and $v$. 
Furthermore, for all $(t,t',x)\in[0,T]^2\times\R^N$ with $t\leq t'$, and $\epsilon>0$ there exists $R>0$ such that, if we denote by $B_R(x)$ the ball in $\R^N$ with radius $R$ and center $x$, we have, for all $(\alpha,v)\in\AR(t)\times\VR(t)$,
\[ P_t[ X^{t,x,\alpha(v),v}_{t'}\in B_R(x))]>1-\epsilon.\]

 We introduce now the value functions of the game: for all $(t,x)\in[0,T]\times\R^N$, we set
\be\label{def:V+} V^+(t,x)=\inf_{\alpha\in\AR(t)}\sup_{\beta\in\BR(t)}J(t,x,\alpha,\beta),\ee
and
\be\label{def:V-} V^-(t,x)=\sup_{\beta\in\BR(t)}\inf_{\alpha\in\AR(t)}J(t,x,\alpha,\beta).\ee

\noindent It is clear that $ V^-(t,x)\leq  V^+(t,x)$. Moreover we have the equivalent formulations
\[ V^+(t,x)=\inf_{\alpha\in\AR(t)}\sup_{v\in\VR_d(t)}J(t,x,\alpha(v),v) \;\mbox{ and }\; V^-(t,x)=\sup_{\beta\in\BR(t)}\inf_{u\in\UR_d(t)}J(t,x,u,\beta(u)).\]

\begin{prop}\label{prop:Lipschitz} The value functions 
$V^+$ and $V^-$ are Lipschitz continuous in $x$, and H\"older continuous  in $t$.
\end{prop}

\begin{proof} The proof is a straightforward consequence of the regularity of $J$. 
\end{proof}

%

%
%
%

 Now we are able to establish a subdynamic programming principle.
\begin{prop}
\label{propppd1}
For all $x\in\R^N$ and $0\leq t_0\leq t_1\leq T$, the following subdynamic programming principle holds:\\
\begin{equation}
\label{ppd1}
 V^+(t_0,x)\leq \inf_{\alpha\in\AR(t_0)}\sup_{v\in\VR_d(t_0)}E_{t_0}[V^+(t_1,X^{t_0,x,\alpha(v),v}_{t_1})].\end{equation}
In particular, $V^+$ is a viscosity  subsolution of the following Hamilton-Jacobi-Isaacs equation
\begin{equation}
\label{eqv+}
\left\{\begin{array}{ll}
V_t+H^+(D^2V,DV,x,t)=0, &(t,x)\in[0,T]\times\R^N,\\\\

V(T,x)=g(x), &x\in\R^N,
\end{array}\right.
\end{equation}
where $\displaystyle H^+(A,\xi,x,t)=\inf_{u\in U}\sup_{v\in V}\left(\frac 12 tr(\sigma\sigma^*(t,x,u,v)A+\langle b(t,x,u,v),\xi\rangle\right)$.
\end{prop}

\begin{proof}
Following \cite{flemingsouganidis} we set 
$\Omega_{t_0,t_1}=\{ \omega:[t_0,t_1]\rightarrow\R^d \mbox{ continuous}\}$.
For $\omega\in\Omega_{t_0}$, we define the pair $(\omega_1,\omega_2)\in\Omega_{t_0,t_1}\times\Omega_{t_1}$ by
\[ \omega_1=\omega|_{[t_0,t_1)}, \; \omega_2=\omega|_{[t_1,T]}-\omega(t_1).\]
The map $\omega\mapsto \pi(\omega):=(\omega_1,\omega_2)$ allows to identify  $\Omega_{t_0}$ with $\Omega_{t_0,t_1}\times\Omega_{t_1}$ and we have $P_{t_0}=P_{t_0,t_1}\otimes P_{t_1}$, where $P_{t_0,t_1}$ is the Wiener measure on $\Omega_{t_0,t_1}$.

For $v\in\VR(t_0)$, we denote by $v_2$ the restriction of $v$ on $[t_1,T]$.
We further set 
$\tilde v_2(\omega_1,\omega_2):=v_2(\omega)$ and remark that, if $v\in\VR_d(t_0)$, then $(\tilde v_2(\omega_1,\cdot),\omega_1\in\Omega_{t_0})$ is a family of processes which belongs to $\VR_d(t_1)$.

 Let us now denote by $V(t_0,t_1,x_0)$ the right-hand side of (\ref{ppd1}).
We fix $\epsilon>0$ and consider $\alpha_0\in\AR(t_0)$ $\epsilon$-optimal for $V(t_0,t_1,x_0)$: for all  $v\in\VR_d(t_0)$, 
\begin{equation}
\label{eq0}
 E_{t_0}[V^+(t_1,X^{t_0,x_0,\alpha_0(v),v}_{t_1})]\leq V(t_0,t_1,x_0)+\epsilon.
\end{equation}
Let $\delta>0$ be the delay of $\alpha_0$. We can suppose that $\delta\leq \epsilon^2\wedge(t_1-t_0)$.
Let $R>0$ be such that, for all $v\in\VR_d(t_0)$,
\[ P_{t_0}\left[X_{t_1-\delta}^{t_0,x_0,\alpha_0(v),v}\in B_R(x_0)\right]>1-\epsilon.\]
For $K\in\N$ large enough, let $\{O_0,\ldots O_K\}$ be a Borel partition of $\R^N$ such that $O_1,\ldots, O_K$ have a radius smaller than  $\epsilon$ 	and $\displaystyle B_{R}(x_0)\subset\cup_{k=1}^KO_k$. Pick, for each $k\in\{ 1,\ldots,K\}$, $x_k\in O_k$ and $\alpha^k\in \AR(t_1)$ $\epsilon$-optimal for $V^+(t_1,x_k)$. 
We fix  $\alpha^0\in\AR(t_1)$ some arbitrary strategy. 
By lemma \ref{lemf}, there exists a measurable, nonanticipative map $F:V_{t_0}\times\Omega_{t_0}\rightarrow{\mathcal C}^0(\R^N)$ such that, for all $v\in\VR_d(t_0)$, $P_{t_0}$-a.s., $X^{t_0,x_0,\alpha_0(v),v}_{t_1-\delta}(\omega)=F(v(\omega),\omega)_{t'}$. We  define  a new strategy $\alpha^\epsilon\in\AR(t_0)$ by
\[
\alpha^\epsilon(v,\omega)_s=\left\{
\begin{array}{ll}
\alpha_0(v,\omega)_s,& \mbox{ if } s\in[t_0,t_1),\\
\alpha^k(v_2,\omega_2)_s,& \mbox{ if } s\in[t_1,T] \mbox{ and }F(v,\omega)\in O_k,\; k\in\{ 0,\ldots,K\}.
\end{array} \right.\]
Set, for all $k\in\{ 0,\ldots,K\}$, $A_k=\{ X^{t_0,x_0,\alpha_0(v),v}_{t_1-\delta}\in O_k\}\subset\FR_{t_0,t_1}$.  Recall that the sets $\{F(v)\in O_k\}$ and $A_k$ differs only by a $P_{t_0}$-null set, and that $P_{t_0}(A_0)\leq \epsilon$.\\
Since $V^+$ is bounded, Lipschitz continuous in $x$ and H\"older in $t$, we get, for all $v\in\VR_d(t_1)$
\begin{equation}
\label{eq1}
 \begin{array}{rl}
\displaystyle  E_{t_0}\left[V^+(t_1,X_{t_1}^{t_0,x_0,\alpha_0(v),v})\right]
= & \displaystyle  E_{t_0}\left[\sum_{k=0}^K\i_{A_k} V^+(t_1,X_{t_1}^{t_0,x_0,\alpha_0(v),v})\right]\\\\

\geq & \displaystyle  E_{t_0}\left[\sum_{k=1}^K\i_{A_k} V^+(t_1,X_{t_1}^{t_0,x_0,\alpha_0(v),v})\right]-\|V^+\|_\infty P_{t_0}(A_0) \\\\

\geq & \displaystyle E_{t_0}\left[\sum_{k=1}^K\i_{A_k}V^+(t_1,x_k)\right]-C\epsilon

\end{array}
\end{equation}
where $C$ denotes a constant which changes from line to line. 

 Now let us come from the left hand side of (\ref{ppd1}):
For any $v\in\VR_d(t_0)$, we can write:
\begin{equation}
\label{eq2}
J(t_0,x_0,\alpha^\epsilon(v),v)=
E_{t_0}\left[\sum_{k=0}^K\i_{A_k}E_{t_0}[g(X_T^{t_1,X_{t_1}^{t_0,x_0,\alpha_0(v),v},\alpha^k(v_2),v_2})ds|\FR_{t_1}]\right].
\end{equation}
But, for all $k\in\{ 1,\ldots,K\}$, for $P_{t_0}$-allmost $\omega\in\Omega_{t_0}$, we have
\begin{equation}
\label{eq3}\begin{array}{rl}
E_{t_0}\left[g(X_T^{t_1,X_{t_1}^{t_0,x_0,\alpha_0(v),v},\alpha^k(v_2),v_2})|\FR_{t_1}\right](\omega)=&
E_{t_1}\left[g(X_T^{t_1,X_{t_1}^{t_0,x_0,\alpha_0(v),v}(\omega_1),\alpha^k(\tilde v_2(\omega_1)),\tilde v_2(\omega_1)})\right]\\\\

=&J(t_1,X_{t_1}^{t_0,x_0,\alpha_0(v),v},\alpha^k(v_2),v_2)(\omega).
\end{array}\end{equation}
Since $J$ is Lipschitz continuous in $x$  and H\"older in $t$ and   $\alpha^k$ is $\epsilon$-optimal for $V^+(t_1,x_k)$, it holds that
\begin{equation}
\label{eq4}\begin{array}{rl}
E_{t_0}\left[\i_{A_k}J(t_1,X_{t_1}^{t_0,x_0,\alpha_0(v),v},\alpha^k(v_2),v_2)\right]\leq &
E_{t_0}\left[\i_{A_k}( J(t_1,x_k,\alpha^k(v_2),v_2)+C\epsilon)\right]\\\\

 \leq & E_{t_0}\left[\i_{A_k} (V^+(t_1,x_k)+C\epsilon)\right].
\end{array}\end{equation}
 Putting together (\ref{eq0})-(\ref{eq4}), we get
\[ J(t_0,x_0,\alpha^\epsilon(v),v)\leq V(t_0,t_1,x_0)+C\epsilon.\]
Taking the sup over $v\in\VR_d(t_0)$ then gives the result.

The proof of the supersolution property from the subdynamic programming is standard (see \cite{flemingsouganidis}). 
\end{proof}

 In a symmetrical  way, we obtain a superdynamic programming principle for $V^-$ and the fact that $V^-$ is a subsolution:

\begin{prop}
\label{propppd2}
For all $x\in\R^N$ and $0\leq t_0\leq t_1\leq T$, the following superdynamic programming principle holds:\\
\[
 V^-(t_0,x)\geq \inf_{\beta\in\BR(t_0)} \sup_{u\in\UR_d(t_0)}E_{t_0}[V^-(t_1,X^{t_0,x,u,\beta(u)}_{t_1})].\]
Therefore $V^-$ is a supersolution in viscosity sense of the following Hamilton-Jacobi-Isaacs equation
\[
\left\{\begin{array}{ll}
V_t+H^-(D^2V,DV,x,t)=0, &(t,x)\in[0,T]\times\R^N,\\\\

V(T,x)=g(x), &x\in\R^N,
\end{array}\right.
\]
where $H^-(A,\xi,x,t)=\sup_{v\in V}\inf_{u\in U}\left(\frac 12 tr(\sigma\sigma^*(t,x,u,v)A+\langle b(t,x,u,v),\xi\rangle\right)$.
\end{prop}

We can now follow  \cite{flemingsouganidis}  to obtain:

\begin{thm}
Under Isaacs' condition:
\[ \forall A\in\SR(N),\ \xi\in\R^N,\ x\in\R^N,\ t\in[0,T], \ H^+(A,\xi,x,t)= H^-(A,\xi,x,t):= H(A,\xi,x,t),\]
 the game has a value $V^+=V^-$ which is the unique solution in viscosity sense of 
\begin{equation}
\label{eqv-}
\left\{\begin{array}{l}
V_t+H(D^2V,DV,x,t)=0, \; (t,x)\in[0,T]\times\R^N,\\\\

V(T,x)=g(x), \; x\in\R^N.
\end{array}\right.
\end{equation}
\end{thm}

\section{Erratum to ``Stochastic Differential Games with Asymmetric Information'' \cite{sdgai}.}\label{sec:3}

The definition  of strategy introduced in Definition \ref{def:strategy} is mainly motivated by a gap in the paper  \cite{sdgai}. This paper deals with two-player, zero-sum differential games in which the players have a private information on the game. The flaw in the paper is {\it not} with this information issue, but with some technicalities arising with the notion of strategy developed there.

In the framework of \cite{sdgai} a {\em strategy} for player I starting at time $t$ is a Borel-measurable map $\alpha : [t,T]\times\CR([t,T],\R^n)\rightarrow U$ for which there exists $\delta>0$ such that, $\forall s\in[t,T], f,f'\in\CR([t,T],\R^n)$, if $f=f'$ on $[t,s]$, then $\alpha(\cdot,f)=\alpha(\cdot,f')$ on  $[t,(s+\delta)\wedge T]$. With this definition, Player I only needs to observe the realization $f$ of the solution of the stochastic differential equation. We can use this notion of strategy to define as in \eqref{def:V+} and \eqref{def:V-} the value functions $V^+$ and $V^-$ via a fixed point argument very close to Lemma \ref{lem:fixpoint}. The issue arises when one tries to prove that these value functions are Lipschitz continuous in space. Indeed, given
a strategy $\alpha$ as above, an adapted control $v\in {\mathcal V}(t)$ and two initial conditions $x$ and $x'$, there seems to be no way to built a new strategy $\alpha'$ such that the solutions $X^{t,x,\alpha,v}$ and $X^{t,y,\alpha',v}$ are sufficiently close (in particular, the idea consisting in choosing $\alpha'(s,f)=\alpha(s,f-x'+x)$ does not seem to work). As a consequence, there is a serious gap in the proof of Lemma 2.2 of \cite{sdgai}. 

In order to correct this, we have to change the notion of strategies and replace it with the one developed in the present note. This implies several changes, that we list below. 

\begin{enumerate} 
\item As in section \ref{defdef}, we assume that we work on the Wiener space 
$\Omega_t=\CR([t,T]\ ,\R^d)$ endowed with the Wiener measure $P_t$  and consider, for all initial time $t\in[0,T]$,  the canonical process $(B_s(\omega)=\omega(s), s\in[t,T])$. The filtration $(\FR_{t,s},t\leq s)$ is the one  generated by the canonical process.  

\item The definition of strategies (Definition 2.2 in  \cite{sdgai}) must be replaced by the one in Definition \ref{def:strategy}. This new notion of strategy must also be used in the definition of random strategies defined in \cite{sdgai}, p. 5. 
%
%
%

\item The fixed point (Lemma 2.1 in  \cite{sdgai}) has to be replaced by Lemma \ref{lem:fixpoint}. 

%


\item The Lipschitz continuity of the value functions (Lemma 2.2 in \cite{sdgai}) becomes straightforward because one can now use the same strategy for different initial positions and get an estimate as in Proposition \ref{prop:Lipschitz}. 

\item In the proof of Proposition 3.1 in \cite{sdgai}, the construction of the strategy has to be modified as follows: by Lemma \ref {lemf}, there exists a measurable map $F:U_t\times\Omega\rightarrow\R^N$, such that, for all $u\in U_t$, it holds that
\[ X^\epsilon_{t_1-\delta}=F(u,\cdot), P\mbox{-a.s.} .\]
Now, for $l\in L, l=(l_0,\ldots,l_M)$, we define $(\beta^\epsilon_j)^l\in\BR(t_0)$ by: $\forall (u, \omega)\in U_{t_0}\times \Omega_{t_0}, \forall t\in[t_0,T],$
\[\begin{array}{l}
\qquad (\beta^\epsilon_j)^l(t,u,\omega)=\left\{\begin{array}{ll}
\beta^\epsilon(v,\omega)_t, \mbox{if } t\in[t_0,t_1),\\
\beta^{m,l_m}_j(u|_{[t_1,T]},\omega_2)& \mbox{if } t\in[t_1,T] \mbox{ and } F(u,\omega)\in E_m.
\end{array}\right.
\end{array}\]
We set $\bar\beta^\epsilon_j:=((\beta^\epsilon_j)^l;s^l_j,l\in L)\in\BR(t_0)$, and finally $\hat\beta^\epsilon=(\bar\beta^\epsilon_1,\ldots,\bar\beta^\epsilon_J)$. Then we can check as in \cite{sdgai} that $\hat\beta^\epsilon$ gives the subdynamic programming. 

\item In the proof of the Corollary 3.1 in \cite{sdgai}, the strategy  has to be changed in the following way: we set $\beta_0(v,\omega)_t=v_0$ for all $(t,v,\omega)\in[t_0,T]\times V_{t_0}\times\Omega_{t_0}$.
\end{enumerate}




\begin{thebibliography}{}

\end{thebibliography}


\begin{thebibliography}{abc99xyz}



\bibitem{buckli} Buckdahn R., Li J. {\it Stochastic differential games and viscosity solutions of Hamilton-Jacobi-Bellman-Isaacs equations.} SIAM J. Control Optim. 47 (2008), no. 1, 444-475.


\bibitem{sdgai} Cardaliaguet P.,  Rainer C. {\it Stochastic Differential Games with Asymmetric Information.} Appl. Math. Optim. (2009) 59, 1-36.

\bibitem{ELKA72} Elliot N.J. \& Kalton N.J. (1972) 
{\it The existence of value in differential games} 
Mem. Amer. Math. Soc., 126.


\bibitem{flemingsouganidis} Fleming W.H., Souganidis P.E. {\it On the existence of value functions of two-player, zero-sum stochastic differential games.} Indiana Univ. Math. J.38(2), 293-314 (1989)

\bibitem{Nutz} Nutz M. {\it Pathwise Construction of Stochastic Integrals.}
Electronic Communications in Probability, Vol. 17, No. 24, pp. 1-7, 2012. 

\bibitem{ROX69} Roxin E. (1969) 
{\it The axiomatic approach in differential 
games}, J. Optim. Theory Appl. 3, 153-163.


\bibitem{VAR167} Varaiya P. (1967) {\it The existence of solution to a 
differential game}, SIAM J. Control Optim. 5, 153-162.

\end{thebibliography}
\end{document}